\documentclass[11pt]{amsart}
\usepackage[normalem]{ulem}
 
\usepackage[all]{xy}
\usepackage{amsmath}
\usepackage{amsfonts}
\usepackage{amssymb}
\usepackage{mathrsfs}
\usepackage{amsxtra}
\usepackage{amsthm}
\usepackage{amsmath,amssymb,amsfonts,amsthm,color,latexsym}
\usepackage{xcolor}
\usepackage{float}

\theoremstyle{plain}

\newcommand{\R}{\mathbb{R}}
\newcommand{\C}{\mathbb{C}}
\newcommand{\N}{\mathbb{N}}
\newcommand{\Z}{\mathbb{Z}}

\newcommand{\F}{\mathcal{F}}
\newcommand{\B}{\mathcal{B}}

\newcommand{\X}{\mathcal{X}}
\newcommand{\D}{\mathcal{D}}

\newcommand{\sing}{\mathrm{Sing}}

\newcommand{\ord}{\text{ord}}

\newcommand{\cl}[1]{\mathcal{#1}}

\newtheorem{maintheorem}{Theorem}

\newtheorem{Btheorem}{Theorem}

\newtheorem{secondtheorem}{Theorem}

\newtheorem{theorem}{Theorem}[section]

\newtheorem{lemma}[theorem]{Lemma}
\newtheorem{proposition}[theorem]{Proposition}
\newtheorem{corollary}[theorem]{Corollary}

\newtheorem{example}[theorem]{Example}
\newtheorem{remark}[theorem]{Remark}

\usepackage{graphicx}

\begin{document}

\title{On Brian\c{c}on-Skoda theorem for foliations}

\date{\today}
\author[A. Fern\'andez-P\'erez]{Arturo Fern\'andez-P\'erez}
\address[Arturo Fern\'{a}ndez P\'erez] {Department of Mathematics. Federal University of Minas Gerais. Av. Ant\^onio Carlos, 6627 
CEP 31270-901\\
Pampulha - Belo Horizonte - Brazil\\}
\email{fernandez@ufmg.br}

\author[E. R.  Garc\'{i}a Barroso]{Evelia R. Garc\'{i}a Barroso}
\address[Evelia R. Garc\'{i}a Barroso]{Dpto. Matem\'{a}ticas, Estad\'{\i}stica e Investigaci\'on Operativa\\
Secci\'on de Matem\'aticas\\
Universidad de La Laguna. Apartado de Correos 456. 38200 La Laguna, Tenerife, Spain.}
\email{ergarcia@ull.es}

\author[N. Saravia-Molina]{Nancy Saravia-Molina}
\address[Nancy Saravia-Molina]{Dpto. Ciencias - Secci\'{o}n Matem\'{a}ticas, Pontificia Universidad Cat\'{o}lica del Per\'{u}, Av. Universitaria 1801,
San Miguel, Lima 32, Peru}
\email{nsaraviam@pucp.pe}

\thanks{The first-named and third-named authors were partially supported by PUCP-PERU \textsc{dgi}: 2020-3-0014. The first-named author was  partially supported by FAPERJ-Pronex E-26/010.001270/2016 and CNPq PQ-302790/2019-5.  The second-named and third-named authors were partially supported by the Pontificia Universidad Cat\'{o}lica del Per\'{u} project CAP 2019-ID 717. The first-named and second-named author was supported by the grant PID2019-105896GB-I00 funded by MCIN/AEI/10.13039/501100011033.
}

\begin{abstract}
We generalize Mattei's result relative to the Brian\c{c}on-Skoda theorem for foliations to the family of foliations of the second type. We use this generalization to establish relationships between the Milnor and Tjurina numbers of foliations of second type, inspired by the results obtained by Liu for complex hypersurfaces and we determine a lower bound for the global Tjurina number of an algebraic curve.
\end{abstract}

\maketitle

\section{Introduction and Statement of the results}

The problem of deciding whether an element of a ring  belongs to a given
ideal of the ring is known as the {\it ideal membership} and dates back to works of Dedekind who gave the precised definition of an ideal. Even if we know generators of the ideal, it is not trivial to determine if an element is a member of it. Therefore it is interesting to give sufficient conditions for ideal membership. An important theorem in this line is the Hilbert’s Nullstellensatz: it states that if $I$ in an ideal in the ring of germs of holomorphic functions at $0\in \mathbb C^n$ and $f$ vanishes on the zero locus of $I$ then there is a power of $f$ belonging to $I$. The {\em Brian\c{c}on-Skoda Theorem} can be seen as an {\it effective} version of the Hilbert Nullstellensatz when $I$ is a jacobian ideal. Let us clarify this last statement.
Let $f(x_{1},\ldots,x_{n})\in \C\{x_{1},\ldots,x_{n}\}$ be a non-unit convergent power series. Consider  its jacobian ideal $J(f)=(\partial_{x_{1}}f,\ldots, \partial_{x_{n}}f)$.
According to Wall \cite{Wall} it was Mather who asked about the smallest $r$ for which
$f^{r}\in J(f)$. It was known then that $f$ is an element of the integral closure of $J(f)$, which implies the existence of a power of $f$ belonging to $J(f)$. At that time it was also known, thanks to Saito \cite{Saito}, that if the origin is an isolated critical point of $f$ then $f$ belongs to $J(f)$ iff $f$ is a quasi-homogeneous polynomial. Brian\c con and Skoda \cite{Briancon-Skoda} proved, using analytic results of Skoda, that $f^{n}\in J(f)$. Later, Lipman and Teissier \cite{Lipman-Teissier} gave an algebraic proof of this algebraic statement. 
Subsequently, Brian\c con-Skoda Theorem has been generalized in different contexts, and has given rise to abundant literature. In Foliation Theory, Mattei  proved

\begin{Btheorem}(\cite[Th\'eor\`eme C]{Mattei}) \label{Thi:}
Let $\F$ be a non-dicritical  generalized curve holomorphic foliation at $(\C^2,p)$ given by  $\omega=P(x,y)dx + Q(x,y)dy$. If $f(x,y)=0$ is the reduced curve of total union of separatrices of $\F$ then  $f^2$ belongs to the ideal $(P,Q)$.
\end{Btheorem}

In this paper, we extend Theorem \ref{Thi:} to the family of second type foliations (perhaps dicritical) and show (see Example \ref{Fk}) that it is essential that the foliation be of the second type.

\begin{maintheorem}\label{th:B-Sth}
Let $\F$ be a germ of a second type holomorphic foliation at $(\C^2,p)$ induced by $\omega=P(x,y)dx+Q(x,y)dy$, where $P,Q\in\C\{x,y\}$, and let $F=f/h$ be a reduced balanced equation of separatrices for $\F$. Then $f^2$ belongs to ideal $(P,Q)$.  
\end{maintheorem}

In Section \ref{sec:preliminaries} we introduce  all the notions and tools necessary to prove Theorem \ref{th:B-Sth}. We are inspired by Mattei's proof but to extend it to the dicritical case we use the characterizations of the dicritical second type foliations given by Genzmer in \cite{genzmer}. The proof of Theorem \ref{th:B-Sth} is given in Section \ref{sec: proof}. In Section \ref{sec:MT}, we obtain relationships between the Milnor number, $\mu_p(\F)$, and the Tjurina number, $\tau_p(\F,\B_0)$, of the foliation $\F$ with respect to the zero divisor $\B_0$ of a balanced divisor of separatrices $\B=\B_0-\B_{\infty}$ of $\F$, inspiring us to do so in the work of Liu \cite{Liu} for complex hypersurfaces. More precisely, if $\mathcal{P}^{\F}$ is a generic polar curve of $\F$, $\nu_p(.)$ denotes the algebraic multiplicity of a curve and $i_p(.,.)$ denotes the intersection multiplicity of two curves then we get

\begin{secondtheorem}\label{cota}
Let $\F$ be a singular holomorphic foliation of second type at $(\C^2,p)$. Let $\B=\B_0-\B_{\infty}$ be a balanced divisor of separatrices for $\F$. Then 
\begin{equation}\label{eq:cota}
\frac{(\nu_p(\B_0)-1)^2+\nu_p(\B_{\infty})-i_p(\mathcal{P}^{\F},\B_{\infty})-i_p(\B_0,\B_{\infty})}{2} \stackrel{(*)}\leq \frac{\mu_p(\F)}{2} \leq \tau_p(\F,\B_0),
\end{equation} 
and the equality $(*)$ holds if $\F$ is a  generalized curve foliation and $\B_0$ is defined by a germ of semi-homogeneous function at $p$. 
Moreover, if $\B_{\infty}=\emptyset$, then 
\[\frac{\nu_p(\F)^2}{2}\leq \frac{\mu_p(\F)}{2} \leq \tau_p(\F,\B_0).\]
\end{secondtheorem}
\par Finally, as consequence of Theorem \ref{cota}, in Section \ref{global_tjurina}, we obtain a lower bound for the global Tjurina number of an algebraic curve.

 \section{Preliminaries} 
 \label{sec:preliminaries}

 \par Let $\F$ be a germ of singular holomorphic foliation at $(\C^2,p)$, in local coordinates $(x,y)$ centered at $p$, the foliation is given by a holomorphic 1-form
\begin{equation}
\label{oneform}
\omega=P(x,y)dx+Q(x,y)dy,
\end{equation}
or by its dual vector field
\begin{equation}
\label{vectorfield}
v = -Q(x,y)\frac{\partial}{\partial{x}} + P(x,y)\frac{\partial}{\partial{y}},
\end{equation}
where  $P(x,y), Q(x,y)   \in {\mathbb C\{x,y\}}$ are relatively prime, where $\mathbb C\{x,y\}$ is the ring of complex convergent power series in two variables. 
 The \textit{algebraic multiplicity} of $\F$, denoted by $\nu_p(\F)$, is the minimum of the orders $\nu_p(P)$, $\nu_p(Q)$ at $p$ of the coefficients of $\omega$. 
\par We say that $C: f(x,y)=0$, with  $f(x,y)\in  \mathbb{C}\{x,y\}$,  is an $\F$-\emph{invariant} curve if $$\omega \wedge d f=(f.h) dx \wedge dy,$$ for some  $h\in \mathbb{C}\{x,y\}$.  A {\em separatrix} of $\F$ is an irreducible $\F$-invariant curve.  Denote by $Sep_p(\F)$ the set of all separatrices of $\F$ through $p$. If $Sep_p(\F)$ is a finite set then we  say that the foliation $\F$ is {\em non-dicritical} and we call {\it total union of separatrices} of $\F$ to the union of all elements of  $Sep_p(\F)$. Otherwise we will say that $\F$ is a {\em dicritical} foliation.

A point $p \in \C^{2}$  is a {\em reduced} or {\em simple} singularity for $\F$  if the
 linear part ${\rm D} v(p)$ of the vector field $v$ in \eqref{vectorfield}  is non-zero and  has eigenvalues $\lambda_1, \lambda_2 \in \C$ fitting in one of the two following cases:
\begin{enumerate}
\item  $\lambda_1 \lambda_2 \neq 0$ and $\lambda_1 / \lambda_2 \not \in \mathbb{Q}^+$  
(in which case we say that $p$ is a {\em non-degenerate} or {\em complex hyperbolic} singularity). \smallskip
\item $\lambda_1 \neq 0$ and  $\lambda_2= 0$ \ (in which case we say that $p$ is a {\em saddle-node} singularity).
\end{enumerate}

The reduction process of the singularities of a codimension one singular foliation over an ambient space of dimension two was achieved by Seidenberg  \cite{seidenberg}. 

A singular foliation $\F$ at $(\C^2,p)$ is a \textit{generalized curve foliation} if it has no saddle-nodes in its reduction process of singularities. This concept was defined by Camacho-Lins Neto-Sad \cite[Page 144]{CLS}. In this case, there is a system of coordinates  $(x,y)$ in which $\F$ is induced by the equation
\begin{equation}
\label{non-degenerate}
\omega=x(\lambda_1+a(x,y))dy-y(\lambda_2+b(x,y))dx,
\end{equation}
where $a(x,y),b(x,y)  \in {\mathbb C}\{x,y\}$ are non-units, so that  $Sep_p(\F)$ is formed by two
transversal analytic branches given by $\{x=0\}$ and $\{y=0\}$. In the case (2), up to a formal change of coordinates, the  saddle-node singularity is given by a 1-form of the type
\begin{equation}
\label{saddle-node-formal}
\omega = x^{k+1} dy-y(1 + \lambda x^{k})dx,
\end{equation}
where $\lambda \in \mathbb{C}$ and $k \in \mathbb{Z}_{>0}$ are invariants after formal changes of coordinates (see \cite[Proposition 4.3]{martinetramis}).
The curve $\{x=0\}$   is an analytic separatrix, called {\em strong}, whereas $\{y=0\}$  corresponds to a possibly formal separatrix, called {\em weak} or {\em central}. 

\par Let $\F$ be a  foliation at $(\C^2,p)$, given by a $1$-form as in (\ref{oneform}), with reduction process $\pi:(\tilde{X},\D)\to (\C^{2},p)$ and let $\tilde{\F} = \pi^{*} \F$ be the strict transform of  $\F$. Denote by $ \sing(\cdot)$ the set of singularities of a foliation.
A saddle-node  singularity $q \in \sing(\tilde{\F})$ 
is said to be a \textit{tangent saddle-node} if  its   weak separatrix is contained in the exceptional divisor $\D$, that is, the weak separatrix is an irreducible component of $\D$.

\par  A foliation   is  \textit{in the second class} or is \textit{of second type} if there
are no    tangent saddle-nodes in its reduction process of singularities. This notion was studied   by Mattei-Salem \cite{mattei} in the non-dicritical case and by Genzmer \cite{genzmer} for arbitrary foliations.\\

\par For a fixed reduction process of singularities $\pi:(\tilde{X},\D)\to(\C^{2},p)$ for $\F$, a component  $D \subset \D$ can be:
\begin{itemize}
\item {\em non-dicritical}, if $D$ is $\tilde{\F}$-invariant. In this case, $D$ contains a finite number of simple singularities. Each non-corner singularity of $D$ carries a separatrix   transversal to $D$, whose projection by $\pi$ is a curve in $Sep_{p}(\F)$. Remember that a corner singularity of $D$ is an intersection point of $D$ with other irreducible component of $\D$.
\item {\em dicritical}, if $D$ is not $\tilde{\F}$-invariant. The reduction process
of singularities gives that $D$ may intersect only non-dicritical components of $\D$ and  $\tilde{\F}$ is everywhere transverse to $D$. The $\pi$-image of a local leaf of $\tilde{\F}$ at each non-corner point of $D$ belongs to $Sep_{p}(\F)$.
\end{itemize}

\par Denote by   $Sep_{p}(D) \subset Sep_{p}(\F)$ the set of separatrices whose strict transforms
 by $\pi$ intersect the
component $D \subset \D$. If $B \in Sep_{p}(D)$ with $D$ non-dicritical, $B$ is said to be \textit{isolated}. Otherwise, it is said to be a \textit{dicritical separatrix}.
This determines the   decomposition $Sep_{p}(\cl{F}) = Iso_{p}(\F) \cup Dic_{p}(\F)$, where notations are self-evident.
The set $Iso_{p}(\F)$  is finite and  contains all   purely
formal separatrices. It   subdivides further    in two classes:
 \textit{weak} separatrices --- those arising from the weak separatrices of saddle-nodes --- and \textit{strong} separatrices --- corresponding to strong separatrices
of saddle-nodes and separatrices of non-degenerate singularities. On the other hand, if  $Dic_{p}(\F)$ is non-empty  then it is an infinite set of analytic separatrices.
Observe that a foliation  $\F$ is {\em  dicritical}
  when $Sep_{p}(\F)$ is infinite, which is equivalent to saying that $Dic_{p}(\F)$ is non-empty. Otherwise, $\F$ is {\em non-dicritical}.

Throughout the text, we would rather adopt the language of \textit{divisors} of formal curves.
More specifically, a \textit{divisor of separatrices} for a foliation $\F$ at $(\C^2,p)$ is
a formal sum
\begin{equation}\label{divisor}
\B = \sum_{B \in \text{Sep}_{p}(\F)} a_{B} \cdot B, 
\end{equation}
where the coefficients $a_{B} \in \mathbb{Z}$ are zero except for finitely many $B \in Sep_{p}(\F)$. The set of separatrices $\{B\;:\; a_{B}\neq 0\}$ appearing in \eqref{divisor} is called the \emph{support} of the divisor $\B$ and it is denoted by $\hbox{\rm supp}(\B)$. The \emph{degree} of the divisor $\B$ is by definition $\deg \B=\sum_{B\in \hbox{\rm supp}(\B) }a_{B}$.
We denote by $Div_{p}(\F)$ the set of all these divisors of separatrices, which turns into a group with the canonical additive structure.
We follow  the usual terminology and notation:
\begin{itemize}
\item $\B \geq 0$ denotes an \textit{effective} divisor, one whose  coefficients are all  non-negative;
\item   there is a unique decomposition $\B = \B_{0} - \B_{\infty}$, where $\B_{0}, \B_{\infty} \geq 0$ are respectively the \textit{zero}
and \textit{pole} divisors of $\B$;
\item the \textit{algebraic multiplicity} of   $\B$ is
$\nu_{p}(\B)=\displaystyle\sum_{B \in \hbox{\rm supp}(\B) } \nu_{p}(B).$
\end{itemize}

Following  \cite[page 5]{genzmer} and \cite[Definition 3.1]{Genzmer-Mol}, we define  a \emph{balanced divisor of separatrices}  for $\F$  as a divisor of the form

\[ \B \ = \
\sum_{B\in {\rm Iso}_p(\F)} B+ \sum_{B\in {\rm Dic}_{p}(\F)}\ a_{B}  \cdot B,
\]

\noindent where the coefficients $a_{B} \in \mathbb{Z}$  are  non-zero except for finitely many $B \in
Dic_{p}(\F)$, and, for each  dicritical  component $D \subset \D$,
 the following equality is respected:
\[\sum_{B \in {\text{Dic}_{p}(D)}}a_{B} = 2- val(D).\]

  The integer $val(D)$ stands for the {\em valence} of a component $D \subset \D$ in the reduction process of singularities, that is,  it is the   number of  components of $\D$ intersecting $D$ other from $D$ itself.

The notion of balanced divisor of separatrices  generalizes, to dicritical foliations, the notion of total union of separatrices for non-dicritical foliations.

 A balanced divisor $\B=\sum_{B}a_{B}B$ of separatrices of $\F$ is called \textit{primitive} if, $a_{B}\in \{-1,1\}$ for any $B\in \hbox{\rm supp}(\B)$.
A \textit{balanced equation of separatrices} is a formal meromorphic function ${F}(x,y)$ whose associated divisor $\B=\B_0-\B_{\infty}$ is a balanced divisor. A balanced equation is \textit{reduced} or \textit{primitive} if the same is true for the underlying divisor.

\par By  \cite[Proposition 2.4]{genzmer} we have 
\begin{equation}\label{G-nst}
\nu_{p}(\F)=\nu_{p}(\B)-1+\xi_p(\F)
\end{equation} and
\begin{equation}\label{genzmer:1}
\hbox{\rm $\F$ is a second type foliation  if, and only if, $\;\nu_{p}(\F)=\nu_{p}(\B)-1,\;\;\;$}
\end{equation}

\noindent where $\B$ is a balanced divisor of separatrices for $\F$ and $\xi_p(\F)$ is the tangency excess of $\F$ at $p$ (see \cite[Definition 2.3]{FP-GB-SM}).

 \section{Proof of Theorem \ref{B-Sth}}
 \label{sec: proof}

Let $\F$ be a germ of a singular holomorphic foliation at $(\C^2,p)$ induced by $\omega:=P(x,y)dx+Q(x,y)dy$, where $P,Q\in\C\{x,y\}$, and consider a blow-up $\sigma:(\tilde{\mathbb{C}}^2,\tilde{D})\to (\C^2,p)$ centered at $p$ with $\tilde{D}=\sigma^{-1}(p)$ and let $\tilde{\F}=\sigma^{*}(\F)$ be the strict transform of $\F$ by $\sigma$. 
  Let $\X_{\tilde{\F}}$ be the sheaf of (holomorphic) vector fields tangent to $\tilde{\F}$ and $H^{1}(\tilde{D},\X_{\tilde{\F}})$ the first cohomology group of $\X_{\tilde{\F}}$ on $\tilde{D}$·
 
\par We have the following lemma, which generalizes \cite[Lemme 2.2.1]{Mattei_inventions} to dicritical blow-ups.
\begin{lemma}\label{lemma_1}

With the above notation, we have 
\[\dim H^{1}(\tilde{D},\X_{\tilde{\F}})=\frac{(\nu_p(\F)-\epsilon_p(\F)-1)(\nu_p(\F)-\epsilon_p(\F)-2)}{2},\]
where 
\begin{equation*}
\epsilon_p(\F)=
\begin{cases}
0 & \text{if $\sigma$ is non-dicritical}
\medskip \\
1 & \text{if $\sigma$ is dicritical}.
\end{cases}
\end{equation*}
\end{lemma}
\begin{proof}
Let $x_1=x$, $y_1=\frac{y}{x}$ and $x_2=\frac{x}{y}$, $y_2=y$ be the local coordinates of $\tilde{\C}^2$. Let $V_1=\{(0,y_1):|y_1|<2\}$ and $V_2=\{(x_2,0):|x_2|<2\}$, we have $\tilde{D}=V_1\cup V_2$ and so that $\mathcal{V}=\{V_{i}\}_{i=1,2}$ is an open covering of $\tilde{D}$. Let $v=-Q(x,y)\partial_{x}+P(x,y)\partial_{y}$ be the vector field  defining $\F$, up some calculations, we obtain that $\sigma^{*}(\F)$ is given by the vector field $v^{(1)}$ satisfy over each chart of $\tilde{\C}^2$:
\begin{equation}\label{campos_1}
X_1=\frac{1}{x_1^{\nu_p(\F)-\epsilon_p(\F)}}\cdot v^{(1)},\,\,\,\,\,\,\,\,\,\,\,\,\,\,\,X_2=\frac{1}{y_2^{\nu_p(\F)-\epsilon_p(\F)}}\cdot v^{(1)}.
\end{equation}
Thus $\X_{\tilde{\F}}(V_i)=\mathcal{O}(V_i)\cdot X_i$ for each $i=1,2$. Set $V_{12}:=V_1\cap V_2=\{(0,y_1):\frac{1}{2}<|y_1|<2\}$. Every section of $\X_{\tilde{\F}}(V_{12})$ can be written as 
$g(x_1,y_1)\cdot X_1$, where 
\begin{equation}\label{function_1}
g(x_1,y_1)=\sum_{(i,j)\in\N\times\Z}g_{ij}x_1^i y_1^{j}
\end{equation} 
is a convergent power series on 
\begin{equation}
\frac{1}{2}<|y_1|<2\,\,\,\,\,\text{and}\,\,\,\,\,\,\, |x_1|<\Gamma_g(y_1),
\end{equation}
 for some continuous function $\Gamma_g:V_{12}\to\R_{+}$. Since elements of $\X_{\tilde{\F}}(V_1)$ can be generate with series of type (\ref{function_1}) whose terms $(i,j)$ with $j<0$ are all zeros, we can consider
$\X_{\tilde{\F}}(V_1)\subset \X_{\tilde{\F}}(V_{12})$.
With respect to $\X_{\tilde{\F}}(V_2)$, their elements are generated by convergent power series of the form 
\[k(x_2,y_2)=\sum_{(\alpha,\beta)\in\N^2}k_{\alpha\beta}x_2^{\alpha}y_2^{\beta}.\]
Therefore, using (\ref{campos_1}), we get
\[g(x_1,y_1)=\frac{1}{y^{\nu_p(\F)-\epsilon_p(\F)}_1}k(\frac{1}{y_1},x_1 y_1),\]
assuming the convergency of $k(x_2,y_2)$ on $\frac{1}{2}<|x_2|<2$ and $|y_2|<\Gamma_k(x_2)$, for some continuous function $\Gamma_k:V_{12}\to\R_{+}$.
Hence, $\X_{\tilde{\F}}(V_2)$ can be generated by power series of type (\ref{function_1}) whose terms $(i,j)$ with $j>i-\nu_p(\F)+\epsilon_p(\F)$ are all zeros and so that $\X_{\tilde{\F}}(V_2)\subset \X_{\tilde{\F}}(V_{12})$. Then, applying Leray's theorem, we obtain
\[H^{1}(\tilde{D},\X_{\tilde{\F}})=\check{H}^{1}(\mathcal{V}, \X_{\tilde{\F}})=\frac{\X_{\tilde{\F}}(V_{12})}{\X_{\tilde{\F}}(V_1)+\X_{\tilde{\F}}(V_2)}.\]
Thus, a basis for $H^{1}(\tilde{D},\X_{\tilde{\F}})$ is given by the sections 
\begin{equation}\label{section_1}
X_{ij}=x_1^{i}y_1^{j}X_1\in\X_{\tilde{\F}}(V_{12})\,\,\,\,\,\,\text{such that}\,\,\,\,\,\,\,\,i\geq 0,\,\,\,\,\,\text{and}\,\,\,\,\,\,\,i-\nu_p(\F)+\epsilon_p(\F)<j<0. 
\end{equation}
In particular, the dimension of $H^{1}(\tilde{D},\X_{\tilde{\F}})$ is $\frac{(\nu_p(\F)-\epsilon_p(\F)-1)(\nu_p(\F)-\epsilon_p(\F)-2)}{2}$.
\end{proof}
Let $\pi:=\sigma_1\circ\ldots\circ\sigma_{\ell}:(\tilde{X},\mathcal{D})\to(\C^2,p)$ be a reduction of singularities of $\F$ at $p\in\C^2$. Denote by $F=f/h$  a reduced balanced equation of separatrices for $\F$ and by $Z_0$ and $Z_{\infty}$  the respective strict transforms by $\pi$ of the curves $\{f=0\}$ and $\{h=0\}$. Let $\tilde{\F}=\pi^{*}(\F)$ be the strict transform of $\F$ by $\pi$. Let $\X_{\tilde{\F}}$ be the sheaf of vector fields tangent to $\tilde{\F}$ and  let $\mathcal{X}_{Z_{0}}$ be the sheaf of vector fields tangent to the divisor $\D$ and to $Z_0$.
\begin{proposition}\label{prop_1}
Let $\F$ be a germ of a second type holomorphic foliation at $(\C^2,p)$ and let $F=f/h$ be a reduced balanced equation of separatrices for $\F$. Put $\varphi=f\circ \pi$.
Then, the morphism 
$$[\varphi]\cdot:H^{1}(\D,\X_{\tilde{\F}})\to H^{1}(\D,\X_{\tilde{\F}}),\,\,\,\,\,\,[Y_{ij}]\longmapsto[\varphi\cdot Y_{ij}]$$
is identically zero. 
\end{proposition}
\begin{proof}
We will prove by induction on the number $\ell$ of blow-ups needed to obtain the reduction of singularities of $\F$. If $\ell=1$, then $H^{1}(\tilde{D},\X_{\tilde{\F}})$ is of finite dimension by Lemma \ref{lemma_1} and it follows from (\ref{section_1}) that 
\[
X_{ij}=x_1^{i}y_1^{j}X_1\,\,\,\,\,\,\text{such that}\,\,\,\,\,\,\,\,(i,j)\in I=\{(i,j):i\geq 0,\,\,\,\,\text{and}\,\,\,\,\,i-\nu_p(\F)+\epsilon_p(\F)<j<0\}
\]
induce a basis for $H^{1}(\tilde{D},\X_{\tilde{\F}}),$ where $x_1=x$, $y_1=\frac{y}{x}$ and $x_2=\frac{x}{y}$, $y_2=y$ are the local coordinates of the blow-up $\tilde{X}$ and $X_1$ is as (\ref{campos_1}). Therefore the sections of the form
\begin{equation}\label{section_2}
x_1^{i}y_1^{j}X_1\,\,\,\,\,\,\text{such that}\,\,\,\,\,\,\,\,i\geq 0,\,\,\,\,j\geq 0\,\,\,\,\,\text{or}\,\,\,\,\,\,\,j\leq i-\nu_p(\F)+\epsilon_p(\F)
\end{equation} 
are elements of $\check{B}(\D,\X_{\tilde{\F}})$ (i.e. 1-coboundary of $\X_{\tilde{\F}}$). Since $\F$ is of second type, $\nu_p(\F)=\nu_p(F)-1$, which implies that $\nu_p(f)=\nu_p(h)+\nu_p(\F)+1$. In particular, $\varphi=f\circ\pi\in(x_1^{\nu_p(h)+\nu_p(\F)+1})$. Hence the sections $\varphi\cdot X_{ij}$ with $(i,j)\in I$ are elements of $\check{B}(\D,\X_{\tilde{\F}})$ and the proof of proposition ends for $\ell=1$.
\par Now, for the general case, we use the exact sequence (see \cite[page 312]{Mattei_inventions}):
\[\xymatrix{
0 \ar[r] & H^{1}(\tilde{D},\X_{\tilde{\F}^{1}}) \ar[r]^{\rho} & H^{1}(\D,\X_{\tilde{\F}}) \ar[r]^{\psi} & H^{1}(D',\X_{\tilde{\F}}) \ar[r] & 0\\
}\]
where $\tilde{D}=\sigma_1^{-1}(p)$, $\tilde{\F}^{1}$ is the strict transform of $\F$ by $\sigma_1$, $D'$ is the union of irreducible components of $\D$ different of $\tilde{D}$, $\psi$ is the restriction morphism and $\rho$ is the morphism induced by the natural inclusion of $\tilde{D}$ in $\D$. Finally, since the following diagram is commutative 
\[\xymatrix{
0 \ar[r] & H^{1}(\tilde{D},\X_{\tilde{\F}^{1}}) \ar[r]^{\rho} \ar[d]^{[f\circ\sigma_1] \cdot}& H^{1}(\D,\X_{\tilde{\F}}) \ar[r]^{\psi} \ar[d]^{[\varphi] \cdot} & H^{1}(D',\X_{\tilde{\F}}) \ar[r] \ar[d]^{[f\circ\sigma_2\circ\ldots\circ\sigma_{\ell}] \cdot}  & 0\\
0 \ar[r] &  H^{1}(\tilde{D},\X_{\tilde{\F}^{1}}) \ar[r]^{\rho}& H^{1}(\D,\X_{\tilde{\F}}) \ar[r]^{\psi} & H^{1}(D',\X_{\tilde{\F}}) \ar[r] & 0}\]
we get $[\varphi]\cdot$ is identically zero, because $[f\circ\sigma_1] \cdot$ and $[f\circ\sigma_2\circ\ldots\circ\sigma_{\ell}] \cdot$ are morphisms identically zero by the first step of the proof and induction hypothesis, respectively. 
\end{proof}

\par Now, we prove our main result. 
\begin{maintheorem}
\label{B-Sth}
Let $\F$ be a germ of a second type holomorphic foliation at $(\C^2,p)$ induced by $\omega=P(x,y)dx+Q(x,y)dy$, where $P,Q\in\C\{x,y\}$, and let $F=f/h$ be a reduced balanced equation of separatrices for $\F$. Then $f^2$ belongs to ideal $(P,Q)$.  
\end{maintheorem}
\begin{proof}
 Let $\pi:(\tilde{X},\mathcal{D})\to(\C^2,p)$ be a reduction of singularities of $\F$ and $\tilde{\F}=\pi^{*}(\F)$ be the strict transform of $\F$ by $\pi$. According to Genzmer \cite[Proposition 3.1]{genzmer}, since $\F$ is of second type, we have the exact sequence of sheaves 
 \[\xymatrix{
0 \ar[r] & \X_{\tilde{\F}} \ar[r] & \X_{Z_{0}} \ar[r]^{\pi^{*}(\frac{\omega}{F})} & \mathcal{O}(-Z_{\infty}) \ar[r] & 0,\\
}\]
where $\X_{\tilde{\F}}$ be the sheaf of vector fields tangent to $\tilde{\F}$ and let $\mathcal{X}_{Z_{0}}$ be the sheaf of vector fields tangent to the divisor $\D$ and to $Z_0$. Then, there exists a covering of $\D$ by open subsets $V_i\subset \tilde{X}$ and holomorphic vector fields $X_i\in \X_{Z_{0}}(V_i)$ such that 
\[\pi^{*}\left(\frac{\omega}{F}\right)(X_i)=h\circ \pi,\,\,\,\,\,\,\,\mathcal{O}(-Z_{\infty})=(h\circ\pi)\mathcal{O},\] which implies that
 \begin{equation}\label{eqq_1}
 \pi^{*}(\omega)(X_i)=(F\circ\pi) \cdot (h\circ \pi)=f\circ\pi.
 \end{equation}
Let $X_{ij}:=X_i-X_j$. It follows from (\ref{eqq_1}) that $\pi^{*}(\omega)(X_{ij})=0$.
Hence
$X_{ij}$ is a 1-cocycle with values over the sheaf $\X_{\tilde{\F}}$ and therefore 
$$[(f\circ\pi)X_{ij}]=0\in H^{1}(\D,\X_{\tilde{\F}}),$$
by Proposition \ref{prop_1}.
Thus, there exists a holomorphic vector field $\tilde{v}$ on $\D$ such that $\tilde{v}|_{V_{i}}=(f\circ\pi)\cdot X_{i}$. Up multiplication by $f\circ\pi$ in (\ref{eqq_1}), we get
$$ \pi^{*}(\omega)(\tilde{v})=(f\circ\pi)^2=f^{2}\circ\pi.$$
The direct image of $\tilde{v}$ by $\pi$ over $(\C^2,p)$ is a holomorphic vector field outside the origin of $\C^2$. The proof ends, by applying Hartogs extension theorem.\end{proof}
\par We note that Theorem \ref{B-Sth} is optimal, in the sense, that the hypothesis on the foliation be of second type cannot be removed. For instance, we have the following example.
\begin{example}\label{Fk} 
Let $\omega=y(2x^8+2(\lambda+1)x^2y^3-y^4)dx+x(y^4-(\lambda +1)x^{2}y^3-x^8)dy$ be a $1$-form defining  a  singular foliation $\F$ at $(\C^2,0)$, which is not of second type and $xy=0$ is the equation of an effective divisor of separatrices for $\F$ (see \cite[Example 6.5]{FP-GB-SM}). We claim that $(xy)^2$ does not belong to the ideal generated by the components of $\omega$. In fact, if $P(x,y):=y(2x^8+2(\lambda+1)x^2y^3-y^4)$, $Q(x,y):=x(y^4-(\lambda +1)x^{2}y^3-x^8)$ and we suppose that $(xy)^2=a(x,y)P(x,y)+b(x,y)Q(x,y)$ for some $a(x,y),b(x,y)\in \mathbb C[[x,y]]$ then $4=\ord (xy)^2\geq \min \{\ord (a(x,y)P(x,y), \ord (b(x,y)Q(x,y)\}\geq 5$ which is a contradiction.
\end{example}

The following corollary will be useful in the following section:
\begin{corollary}\label{coro:BS}
Let $\F$ be a germ of a second type holomorphic foliation at $(\C^2,p)$ induced by $\omega=P(x,y)dx+Q(x,y)dy$, where $P,Q\in\C\{x,y\}$, and let $\B$ be a reduced balanced equation of separatrices for $\F$. If $\B_0:f(x,y)=0$ and $\bar f$ is the coset of $f$ modulo $(P,Q)$ then the complex vector spaces $(f,P,Q)/(P,Q)$ and $(\bar f)/(\bar f^2)$ are isomorphic.
\end{corollary}
\begin{proof}
Put $\frak T=(f,P,Q)$. The map $\psi:\frak T\longrightarrow (\bar f)/(\bar f^2)$ given by \[\psi(g_zf+\alpha P+\beta Q)=\overline{g_zf}\;\; \hbox{\rm mod } (\bar f^2)\] is an epimorphism of complex vector spaces. Finally by Theorem \ref{B-Sth} the kernel of $\psi$ equals $(P,Q)$.

\end{proof}

 \section{Milnor and Tjurina numbers after the Brian\c{c}on-Skoda theorem}
 \label{sec:MT}
 Let $\F$ be a singular holomorphic foliation at $(\C^{2},p)$ given by the $1$-form  $\omega:=P(x,y)dx+Q(x,y)dy$. Assume that $\F$ has an isolated singularity at $p$ and consider the jacobian  ideal associated with $\F$ given by $J(\F)=(P,Q)$. Then $\mathcal M(\F):=\C[[x,y]]/J(\F)$ is a finite $\C$-dimensional vector space which dimension is called the {\it Milnor number} of $\F$ and we denote it by $\mu_{p}(\F)$. It is well-known, after \cite{CLS}, that the Milnor number is a topological invariant of the foliation. Let $C: f(x,y)=0$ be an $\F-$invariant reduced curve.  
 Put $\mathcal T(\F,C):=\mathbb  C[[x,y]]/(f,P,Q)$, \noindent where $(\cdot, \cdot, \cdot)$ denotes the ideal generated by three elements in $\mathbb C[[x,y]]$. 
 
\noindent The {\it Tjurina number of} $\F$  {\it with respect to} $C$ is
\[\tau_p(\F,C)=\dim_{\mathbb  C} \mathcal T(\F,C).\]

\par Let $\B$ be a balanced divisor of separatrices for $\F$. Put $\B_{0}: f(x,y)=0$ the zero divisor of $\B$. By definition $\tau_p(\F, \B_{0})\leq \mu_{p}(\F)$. Put $\frak T=(f,P,Q)$.  From the third isomorphic theorem for complex vector spaces we have 

\[\tau_p(\F,\B_{0})=\dim_{\C} \mathbb C[[x,y]]/\frak T=\dim_{\C} \mathcal M(\F)-\dim_{\C} \frak T/J(\F),\] 

so
\begin{equation}\label{diferencia}
\mu_p(\F)-\tau_p(\F,\B_{0})=\dim_{\C} \frak T/J(\F).
\end{equation}

For any $z\in \mathbb C[[x,y]]$ we denote by $\bar z$ the coset of $z$ modulo $J(\F)$ and $\hat z$ its coset modulo $\frak T$. Inspired by Liu \cite{Liu} we consider the exact sequence
\[0\longrightarrow {\rm Ker }\; \sigma \stackrel{i}\longrightarrow \mathcal M(\F) \stackrel{\sigma} \longrightarrow \mathcal M(\F) \stackrel{\delta_{\B}}\longrightarrow  \mathcal T(\F,\B_{0})\longrightarrow 0,\]
where  $i$ is the inclusion map, $\sigma$ is the multiplication by $\bar f$, that is, $\sigma(\bar z)=\overline{zf}$ and $\delta_{\B}(\bar z)=\hat z$. 
Since  $\delta_{\B}$ is surjective, we get 
\begin{equation}\label{diferencia2}
\mu_p(\F)-\tau_p(\F,\B_{0})=\dim_{\C} {\rm Ker }\;\delta_{\B}.
\end{equation}

From \eqref{diferencia2} and the equality $\mu_p(\F)=\dim_{\C} {\rm Ker }\;\sigma+\dim_{\C} {\rm Im }\;\sigma$, we conclude

\begin{equation}
\label{kernel}
\tau_p(\F,\B_{0})=\dim_{\C}{\rm Ker }\;\sigma=\dim_{\C} (J(\F):\B_{0})/J(\F),
\end{equation}
where $(J(\F):\B_{0})=\{z\in \C[[x,y]]\;:\;zf\in J(\F)\}$.

\begin{proposition}\label{Liu_fol}
Let $\F$ be a singular holomorphic foliation of second type at $(\C^2,p)$ given by  the $1$-form $\omega= P(x,y)dx+Q(x,y)dy=0$. Let $\B$ be a balanced divisor of separatrices for $\F$ with $\B_{0}: f(x,y)=0$. Then 
$\tau_p(\F,\B_{0})\leq \mu_p(\F)\leq 2 \tau_p(\F,\B_{0})$. Moreover $\mu_p(\F)= 2 \tau_p(\F,\B_{0})$ if and only if $\ker \sigma=(\bar f)$, where $\bar f$ is the coset of $f$ modulo $(P,Q)$.
\end{proposition}

\begin{proof}
Let us prove the inequality $\mu_p(\F)\leq 2 \tau_p(\F,\B_{0}).$ By Theorem \ref{B-Sth} we get $f^{2}\in J(\F)$, that is, $\overline{f^{2}}=\bar 0\in \mathcal M(\F)$. Hence, we get the inclusion of ideals $\frak T\subseteq {\rm Ker } \; \sigma$.  Moreover we have the following chain of ideals of  $\mathcal M(\F)$:
\[
\mathcal M(\F) \supseteq (\bar f) \supseteq (\bar{f^{2}})=(\bar 0)
\]
where $(\cdot)$ denotes a principal ideal. We also have  the exact sequence: 
\[0\rightarrow {\rm Ker }\; \sigma \cap (\bar f)\stackrel{i} \rightarrow (\bar f) \stackrel{\sigma'} \rightarrow (\bar f) \stackrel{e} \rightarrow (\bar f)/(\bar{f^{2}})\rightarrow 0,
\]
where $i$ is the inclusion map, $\sigma'$ is the multiplication by the coset $\bar f$ and $e$ is the natural epimorphism. We have 
\begin{eqnarray*}
\dim_{\C} {\rm Ker }\;\sigma'+\dim_{\C} {\rm Im }\;\sigma' &=&\dim_{\C} (\bar f)=\dim_{\C} {\rm Ker }\;e+\dim_{\C} {\rm Im }\;e\\
&=&\dim_{\C} {\rm Im }\;\sigma'+\dim_{\C} (\bar f)/(\bar{f^{2}}),
\end{eqnarray*}
so from \eqref{kernel} we get
\[
\dim_{\C} (\bar f)/\bar{(f^{2}})=\dim_{\C} {\rm Ker }\;\sigma'=\dim_{\C} {\rm Ker }\;\sigma \cap (\bar f)\leq 
\dim_{\C} {\rm Ker }\;\sigma=\tau(\F,\B_{0}).
\]

After Corollary \ref{coro:BS} we have $\dim_{\C}(\bar f)/(\bar f^2)=\dim_{\C}\frak T /J(\F)$ and by \eqref{diferencia} we conclude $\mu_p(\F)\leq 2 \tau_p(\F,\B_{0}).$ Finally $\mu_p(\F)=2 \tau_p(\F,\B_{0})$ if and only if ${\rm Ker }\;\sigma \cap (\bar f)={\rm Ker }\;\sigma$, so ${\rm Ker }\;\sigma \subseteq (\bar f)$. We conclude the proof since $\sigma(\bar f)=\bar 0$.
\end{proof}

The {\it intersection multiplicity} of two curves $C:f(x,y)=0$ and $D:g(x,y)=0$ at the point $p$ is by definition $i_p(C,D)=\dim_{\C}\C\{x,y\}/(f,g)$ where $(f,g)$ denotes the ideal of $\C\{x,y\}$ generated by the power series $f$ and $g$.\\

The {\it polar curve} of the singular foliation $\F:\omega=P(x,y)dx+Q(x,y)dy=0 $  at $(\C^2,p)$ with respect to a point $(a:b)$ of the complex projective line $\mathbb P^1(\mathbb C)$ is the analytic  curve $\mathcal{P}^{\F}_{(a:b)}: aP(x,y)+bQ(x,y)=0$. There exists an open Zariski set $U$ of $\mathbb P^1(\mathbb C)$ such that  $\{aP(x,y)+bQ(x,y)=0\;:\; (a:b)\in U\}$ is an equisingular family of plane curves. Any element of this set is called {\it generic polar curve} of the foliation $\F$ and we will denote it by $\mathcal{P}^{\F}$.

A germ of plane curve $C:f(x,y)=0$ of multiplicity $n$ is a {\it semi-homogenous function} at $p$ if and only if $f=f_n+g$ where $f_n$ is a homogeneous polinomial of degree $n$ defining an isolated singularity  at $p$ and $g$ consists of terms of degree at least $n+1$.

\begin{secondtheorem}\label{cota}
Let $\F$ be a singular holomorphic foliation of second type at $(\C^2,p)$. Let $\B=\B_0-\B_{\infty}$ be a balanced divisor of separatrices for $\F$. Then 
\begin{equation}\label{eq:cota}
\frac{(\nu_p(\B_0)-1)^2+\nu_p(\B_{\infty})-i_p(\mathcal{P}^{\F},\B_{\infty})-i_p(\B_0,\B_{\infty})}{2} \stackrel{(*)}\leq \frac{\mu_p(\F)}{2} \leq \tau_p(\F,\B_0),
\end{equation} 
and the equality $(*)$ holds if $\F$ is a  generalized curve foliation and $\B_0$ is defined by a germ of semi-homogeneous function at $p$. 
Moreover, if $\B_{\infty}=\emptyset$, then 
\[\frac{\nu_p(\F)^2}{2}\leq \frac{\mu_p(\F)}{2} \leq \tau_p(\F,\B_0).\]
\end{secondtheorem}
\begin{proof}
By \cite[Proposition 4.2]{FP-GB-SM}, for any singular foliation $\F$ we have  
\begin{eqnarray}\label{eq:1}
    \Delta_p(\F,\B_0)= i_p(\mathcal{P}^{\F},\B_0)+i_p(\B_0,\B_{\infty})-\mu_p(\B_0)-\nu_p(\B_0)+1,
\end{eqnarray}

\noindent where $\Delta_p(\F,\B_0)$ is the excess polar number of $\F$
with respect to $\B_0$.
Since $\F$ is of second type, $\nu_p(\F)=\nu_p(\B)-1=\nu_p(\B_0)-\nu_p(\B_\infty)-1$ by equation (\ref{genzmer:1}), and therefore, from \eqref{eq:1} we get  
\begin{eqnarray}\label{eq:2}
    \Delta_p(\F,\B_0)= i_p(\mathcal{P}^{\F},\B_0)+i_p(\B_0,\B_{\infty})-\mu_p(\B_0)-\nu_p(\F)-\nu_p(\B_\infty).
\end{eqnarray}

On the other hand, after \cite[Theorem A]{Genzmer-Mol}
 we know that $\Delta_p(\F,\B_0)\geq 0$, and equals zero if and only if $\F$ is a  generalized curve foliation. Hence  from \eqref{eq:2} we have
\begin{eqnarray}\label{eq:3}
\mu_p(\B_0)\leq i_p(\mathcal{P}^{\F},\B_0)+i_p(\B_0,\B_{\infty})-\nu_p(\F)-\nu_p(\B_\infty).
\end{eqnarray}
Now, by applying  
\cite[Lemma 4.4]{FP-GB-SM} to $\F$, which is of second type, and by properties of the intersection multiplicity one gets
\begin{eqnarray} \label{eq:4}
i_p(\mathcal{P}^{\F},\B_0)=i_p(\mathcal{P}^{\F},\B_\infty)+\mu_p(\F)+\nu_p(\F),
\end{eqnarray}
so from \eqref{eq:3} and  \eqref{eq:4},   
\begin{eqnarray}\label{eq:5}
\mu_p(\B_0)\leq\mu_p(\F)+i_p(\B_0,\B_{\infty})+i_p(\mathcal{P}^{\F},\B_\infty)-\nu_p(\B_\infty).
\end{eqnarray}
It follows from the definition of the Milnor number, the properties of the intersection multiplicity and \eqref{eq:5} that 
\begin{eqnarray}
(\nu_p(\B_0)-1)^2\leq \mu_p(\B_0)\leq\mu_p(\F)+i_p(\B_0,\B_{\infty})+i_p(\mathcal{P}^{\F},\B_\infty)-\nu_p(\B_\infty).
\end{eqnarray} 
Observe that the first  inequality becomes an equality when $B_0$ is defined by a germ of semi-homogeneous function at $p$ (see \cite{Teissier-Cargese}) and the second inequality is an equality if and only if $\F$ is a generalized curve foliation. 
Finally, the proof ends, up applying Proposition \ref{Liu_fol} 
\begin{equation}
(\nu_p(\B_0)-1)^2+\nu_p(\B_\infty)-i_p(\B_0,\B_{\infty})-i_p(\mathcal{P}^{\F},\B_\infty)\leq \mu_p(\F)\leq 2\tau_p(\F,\B_0).
\end{equation}
\end{proof}

\begin{example}\label{radial}
We illustrate Theorem \ref{cota} with the radial foliation $\F$ given by the $1$-form $\omega=xdy-ydx$. In this case we consider $\B_0=xy(x-y)$ and $\B_{\infty}=x+y.$ We get
$\nu_0(\B_0)=3$, $1=\nu_0(\B_{\infty})=
i_0(\mathcal{P}^{\F},\B_{\infty})=\tau_0(\F,\B_0)$ and  $i_0(\B_0,\B_{\infty})=3$. Hence $\F$ verifies \eqref{eq:cota}.
\end{example}

\begin{remark}
The family of foliations given in \cite[Example 6.5]{FP-GB-SM} are defined by the $1$-form
\[\omega_k=y(2x^{2k-2}+2(\lambda+1)x^2y^{k-2}-y^{k-1})dx+x(y^{k-1}-(\lambda+1)x^2y^{k-2}-x^{2k-2})dy\]

\noindent is a family of dicritical foliations  which are not of second type,   $\B=B_1+B_2$ is an effective balanced divisor of separatrices  for $\F_k$. We get $\nu_0(\F_k)=k$ and $\tau_0(\F_k,\B)=3k-2$. Hence the inequality \[\frac{\nu_p(\F)^2}{2}\leq\tau_p(\F,\B)\]

\noindent fails for all $k\geq 6$. Therefore, in Theorem \ref{cota} the second type hypothesis over $\F$ is essential.

\end{remark}

\section{A lower bound for the global Tjurina number of an algebraic curve}\label{global_tjurina}

\par Let $C$ be a reduced curve of degree $\deg(C)$ in the complex projective plane $\mathbb{P}^2$. Denote by $\tau(C)$ the \textit{global Tjurina number} of the curve $C$, which is the sum of the Tjurina numbers at the singular points of $C$. In this section, under some conditions, we give a lower bound for $\tau(C)$.

\par A holomorphic foliation $\F$ on $\mathbb{P}^{2}$ of degree $d\geq 0$ is a foliation defined by a polynomial 1-form $\Omega=A(x,y,z)dx+B(x,y,z)dy+C(x,y,z)dz,$ where 
$A,B,C$ are complex homogeneous polynomials of degree $d+1$, satisfying two conditions:
\begin{enumerate}
    \item the integrability condition $\Omega\wedge d\Omega =0$,
    \item the Euler condition $Ax+By+Cz=0$.
\end{enumerate}
An algebraic curve $C:f(x,y,z)=0$ is $\F$-invariant if $\Omega\wedge df=f\Theta,$
where $\Theta$ is some polynomial 2-form. 
\par Denote by $\left\lceil z \right\rceil$ the ceiling function evaluated at $z\in \R$, that is, the smallest integer that is greater than or equal to $z\in \R$. We have:

\begin{theorem}\label{teo_global}
Let $\F$ be a holomorphic foliation on $\mathbb{P}^2$ of degree $d$. Suppose that all points $p\in Sing(\F)$ are of second type.
Then
\begin{equation}\label{equa_1}
\left\lceil \frac{d^2+d+1-2\sum_{p\in Sing(\F)}GSV_p(\F,(F_p)_0)}{2}\right\rceil \leq \sum_{p\in Sing(\F)}\tau_p((F_p)_0),    
\end{equation}
where $(F_p)_0$ is the zero  divisor of a balanced equation of separatrices $F_p$ for $\F$ at $p$. In particular, if $C$ is an $\F$-invariant reduced curve in $\mathbb{P}^2$ such that $Sing(\F)\subset C$ and for all $p\in Sing(\F)$, the germ of $C$ at $p$ defines the zero divisor of a balanced equation of separatrices for $\F$ at $p$, then
\begin{equation}\label{equa_2}
    \left\lceil\frac{d^2+d+1-2(d+2)\deg(C)+2\deg(C)^2}{2}\right\rceil\leq \tau(C),
\end{equation}
\end{theorem}
\begin{proof}
Since all points $p\in Sing(\F)$ are of second type, then
\begin{equation}\label{eq_B}
\mu_p(\F)\leq 2\tau_p(\F,(F_p)_0)
\end{equation} by Theorem \ref{cota}. According to \cite[Proposition 6.2 ]{FP-GB-SM}, we have $\tau_p(\F,(F_p)_0)=GSV_p(\F,(F_p)_0)+\tau_p((F_p)_0)$. Hence, up substituting in (\ref{eq_B}), we obtain 
\[\frac{\mu_p(\F)-2 GSV_p(\F,(F_p)_0)}{2}\leq \tau_p((F_p)_0),\,\,\,\,\,\,\text{for all}\,\,\,p\in Sing(\F).\]
The inequality (\ref{equa_1}) is proved by taking sum over all singular points of $\F$, by using $\sum_{p\in Sing(\F)}\mu_p(\F)=d^2+d+1$ (see \cite[Page 28]{Brunella-libro}) and considering the ceiling function. The inequality (\ref{equa_2}) follows from $$\sum_{p\in Sing(\F)\cap C} GSV_p(\F,C)=(d+2)\deg(C)-\deg(C)^2$$ given in \cite[Proposition 4]{index} and considering again the ceiling function.
\end{proof}
\par The following example illustrates Theorem \ref{teo_global}.
\begin{example}
For each $\lambda\in\mathbb{C}$, we consider the 1-form
\[\omega_\lambda=yzdx+\lambda xz dy-(\lambda+1)xy dz,\] which defines a foliation $\F_{\lambda}$ on $\mathbb{P}^2$ of degree one. The curve $C:$ $xyz=0$ has degree three and it satisfies all hypotheses of Theorem \ref{teo_global}. Then 
\[\left \lceil\frac{1^2+1+1-2(1+2)3+2\cdot 3^2}{2}\right \rceil=\left \lceil\frac{3}{2}\right \rceil=2\leq \tau(C)=3,\]
which implies that the inequality (\ref{equa_2}) of Theorem \ref{teo_global} is verified. Observed that we equate the bound given by du Plessis and Wall in \cite[Theorem 3.2]{duPlessis-Wall}.
\end{example}

\textit{Acknowledgement.} The first-named author thanks Universidad de La Laguna for the hospitality during his visit in June 2022.

\end{document}